\documentclass[psamsfonts]{amsart}
\usepackage{amssymb,amsfonts}
\usepackage{enumerate}
\usepackage{mathrsfs}
\usepackage{latexsym}
\usepackage{url}
\usepackage{mathtools}

\newcommand{\Nat}{\mathbb N}
\newcommand{\al}{\alpha}
\newcommand{\om}{\omega}
\newtheorem{thm}{Theorem}[section]
\newtheorem{cor}[thm]{Corollary}
\newtheorem{prop}[thm]{Proposition}
\newtheorem{lem}[thm]{Lemma}

\theoremstyle{definition}
\newtheorem{defn}[thm]{Definition}

\newtheorem{notn}[thm]{Notation}

\newtheorem{Proposition}[thm]{Proposition}

\theoremstyle{remark}

\renewcommand{\wp}{{\mathcal P}}
\makeatletter
\let\c@equation\c@thm
\makeatother
\numberwithin{equation}{section}

\author{Harry Altman and Andreas Weiermann}
\address{Department for Mathematics: Analysis, Logic and Discrete Mathematics: WE16\\
Krijgslaan 281 S8\\
B-9000 Ghent\\
Belgium}
\address{Department of Mathematics\\
2074 East Hall\\
530 Church Street\\
Ann Arbor, MI 48109\\
United States}
\title[Maximum linearizations of lower sets in $\mathbb{N}^m$]{Maximum
linearizations of lower sets in $\mathbb{N}^m$ with application to monomial
ideals}
\date{May 8, 2025}

\begin{document}
\begin{abstract}
We compute the type (maximum linearization) of the well partial order of bounded
lower sets in $\mathbb{N}^m$, ordered under inclusion, and find it is
$\omega^{\omega^{m-1}}$.  Moreover we
compute the type of the set of all lower sets in $\mathbb{N}^m$, a topic studied
by Aschenbrenner and Pong 
in \cite{AP}, and find that it is equal to
\[ \omega^{\sum_{k=1}^{m} \omega^{m-k}\binom{m}{k-1} }+ 1. \]
As a consequence we deduce corresponding lower bounds on effectively given sequences of lower sets and 
effectively given sequences of monomial ideals
in $F[X,Y]$ where $F$ is a field.
\end{abstract}

\maketitle

\section{Introduction}

In this paper we compute the type of several well partial orders.  The
\emph{type} of a well partial order $X$, denoted $o(X)$, is the largest order
type of a well-order extending the order on $X$; this was proven to exist by De
Jongh and Parikh \cite{DJP}. The type $o(X)$ can also be characterized
inductively as the smallest ordinal greater than $o(Y)$ for any proper lower
set $Y$ of $X$.  The theory has been rediscovered several times; the term
``type'' comes from Kriz and Thomas \cite{KT}.  

In this paper we are interested in well partial orders whose elements are lower
sets in the partial order $\mathbb{N}^m$.  
Here, given a partial order $X$, a {\em lower} set in $X$ is a subset $S$ of $X$
that is downward closed, that is, if $s \in S$, $t \in X$, and $t \leq s$,
then $t \in S$; they are also known as {\em initial segments} of $X$. Every
subset $S$ of $X$ is contained in a smallest initial segment of $X$,
called the initial segment of $X$ generated by $S$. An initial segment
generated by a finite subset of $X$ is said to be finitely generated. 

We define:

\begin{defn}
If $X$ is a partial order, we define $I(X)$ to be the poset of lower sets in $X$
ordered under inclusion, and
define $D(X)$ to be the ordered subset of $I(X)$ consisting of the finitely generated
initial segments of $X$.
\end{defn}

Then we are interested in $D(\mathbb{N}^m)$ and $I(\mathbb{N}^m)$.  (Note that
in the case of $\mathbb{N}^m$, we could equivalently define $D(\mathbb{N}^m)$ to
be the set of bounded lower sets, or the set of finite lower sets.)  We prove
the following two theorems:

\begin{thm}
\label{mk}
For $k,m\ge 1$,
\[o(D(\mathbb{N}^m\times k)) = 
\omega^{\omega^{m-1}k}. \]
In particular, \[o(D(\mathbb{N}^m)) = 
\omega^{\omega^{m-1}}. \]
\end{thm}

\begin{thm}
\label{omegaunbounded}
\[ o(I(\mathbb{N}^m)) = 
\omega^{\sum_{k=1}^{m} \omega^{m-k}\binom{m}{k-1} }+ 1. \]
\end{thm}

The last of these questions, that of determining $o(I(\mathbb{N}^m))$, was
asked earlier by Aschenbrenner and Pong \cite{AP}, who provided upper and
lower bounds.  Theorem~\ref{mk} provides an exact answer to this
question when $k=1$.

Theorem~\ref{mk} for $k=1$, the case of $D(\mathbb{N}^m)$, is the ``core'' case,
handled by means of the inductive characterization of $o(X)$ above.  Meanwhile,
the cases of $D(\mathbb{N}^m \times k)$ and $I(\mathbb{N}^m)$ are handled
combinatorially, by using Theorem~\ref{mk} in combination with De
Jongh and Parikh's theorems that $o(X\amalg Y)=o(X)\oplus o(Y)$ and $o(X\times
Y)=o(X)\otimes o(Y)$, where $\oplus$ and $\otimes$ are the natural (or
Hessenberg) sum and product of ordinals.

More specifically, the case of
$D(\mathbb{N}^m \times k)$ is handled by putting together $k$ copies of
$D(\mathbb{N}^m)$, while the case of
$I(\mathbb{N}^m)$ is handled by putting together
$D(\mathbb{N}^m)$ together with $D(\mathbb{N}^C)$, where $C$ ranges over
nonempty subsets of $\{1,\ldots,m\}$.

In Section~\ref{secapp} we will show that the lengths of effectively given sequences of 
lower sets and effectively given sequences of
monomial ideals in $F[X,Y]$ 
are bounded from below by Hardy functions whose levels are determined by the type
of the underlying well partial ordering.

In a future paper \cite{genlower}, we will extend these results to lower sets in
products of larger ordinals as well.

\section{Bounded lower sets in $\mathbb{N}^m$}

In this section we show how to compute $o(D(\mathbb{N}^m))$, proving
Theorem~\ref{mk} for $k=1$.

First, we recall some basic properties of the type:
\begin{prop}[De Jong and Parikh, \cite{DJP}]
Let $X$ and $Y$ be well partial orders.  If $X$ embeds in $Y$, then $o(X)\le
o(Y)$.  Similarly, if there is a weakly increasing surjection from $Y$ onto $X$,
then $o(X)\le o(Y)$.
In particular, if $\le$ and $\le'$ are two well partial orderings on the set
$X$, and $\le'$ extends $\le$, then $o(X,\le')\le o(X,\le)$.  

Also, if $X$ and $Y$ are any two well partial orders, one has
\[ o(X\amalg Y)=o(X)\oplus o(Y)\]
and
\[ o(X\times Y)=o(X)\otimes o(Y),\]
where $\oplus$ and $\otimes$ are the natural (or Hessenberg) sum and product of
ordinals.  As such, if $X$ is a well partial order and $S,T\subseteq X$, then
$o(S\cup T)\le o(S)\oplus o(T)$.
\end{prop}

In order to prove Theorem~\ref{mk}, we apply a similar
lemma (which had been communicated by Schnoebelen and Schmitz to the authors
and which recently was made available on arXiv \cite{Abriola}).

\begin{lem}[\cite{Abriola}]
\label{dlem}
Let $X$ be a well partial order.  Then 
\[o(D(X))\le 2^{o(X)}.\]
\end{lem}

Note that Abriola et al. actually stated their lemma not for $D(X)$,
but rather for the more commonly-studied
$(\wp_{\mathrm{fin}}(X),\leq_{\mathrm{m}})$, where $\wp_\mathrm{fin}(X)$ denotes
the set of finite subsets of $X$ and where we define $S\le_\mathrm{m} T$
if for every $s\in S$, there is some $t\in T$ with $s\le t$.  Of course,
$(\wp_{\mathrm{fin}}(X),\leq_{\mathrm{m}})$ is not actually isomorphic to
$D(X)$, as the former lacks antisymmetry, being only a quasi-order rather than a
partial order; but after quotienting out by equivalences the resulting partial
order is isomorphic to $D(X)$.  So in essence these are the same.

For convenience of the reader we include a proof of Lemma~\ref{dlem}.

\begin{proof}
We use standard arguments from \cite{dianaschmidt}  (following the lines of
\cite{provingterminationfortermrewritingsystems}).  First note that if $o(X)=0$
(i.e. $X$ is empty), the statement is trivial.

Now suppose that $o(X)$ is a limit ordinal. Then $2^{o(X)}$ is a power of
$\omega$, i.e., additively closed.  Given
$a\in X$, let $X^{\not\ge a} := \{x\in X: x\not\ge a\}$, a proper lower subset
of $X$; so $o(X^{\not\ge a}) < o(X)$.  Let $I$ be an element of $D(X)$ and let
\[ S := \{ J\in D(X): J\not\supseteq I\}; \]
we need to show that $o(S) < 2^{o(X)}$.  Take a finite set $A$ such that $I$ is
the downward closure of $A$.  Then for each $J\in S$ we have $J\not\supseteq
A$, so $J\in D(X^{\not\ge a})$ for some $a \in A$.  This shows that $S\subseteq
\bigcup_{a\in A} D(X^{\not\ge a})$ and therefore that $o(S) \le \bigoplus_{a\in
A} o(D(X^{\not \ge a}))$.  By the inductive hypothesis, for each $a\in X$ we
have $o(D(X^{\not\ge a})) \le 2^{o(X^{\not\ge a})} < 2^{o(X)}$.  As $2^{o(X)}$
is additively closed, we obtain $o(S)< 2^{o(X)}$ as desired.

Finally suppose that $o(X)$ is a successor; say $o(X)=\eta+1$. 
Theorem 3.2
of De Jongh and Parikh \cite{DJP} yields a maximal
$x\in X$ with $o(X\setminus\{x\})=\eta$.  So if $I\in D(X)$, then either
$I\in D(X\setminus \{x\})$ or $x\in I$.
Moreover, we have an increasing surjection
\[ J \mapsto J\cup \{x\} : D(X\setminus\{x\}) \to \{I\in D(X): x\in I\}.\]
So, applying the inductive hypothesis, $o(D(X))\le
2^\eta \oplus 2^\eta$.  Since $2^\eta$ contains only a single
distinct power of $\omega$ in its Cantor normal form, one has $2^\eta\oplus
2^\eta=2^{\eta+1}=2^{o(X)}$; thus $o(D(X))\le
2^{o(X)}$.
This completes the proof.
\end{proof}

\begin{cor}
\label{upbd} For $m\ge 1$,
$o(D(\mathbb{N}^m))\leq \om^{\om^{m-1}}$.
\end{cor}

\begin{proof}
One has $o(\mathbb{N}^m)=\omega^m$, so
\[o(D(\mathbb{N}^m))\le 2^{\omega^m}
=(2^\omega)^{\omega^{m-1}}=\omega^{\omega^{m-1}}.\]
\end{proof}

Now we prove the lower bound:

\begin{prop}
\label{lobd} For $m\ge 1$,
$o(D(\mathbb{N}^m))\geq \om^{\om^{m-1}}$.
\end{prop}

\begin{proof}
For a sequence $a=(a_1,\ldots,a_{m-1})$ of length $m-1$ we define
$ord(a)=\om^{m-2}\cdot a_1+\cdots+ \om^0\cdot a_{m-1}$.  For a finite non empty
downward closed subset $F$ in $\mathbb{N}^{m}$ assume that $F$ is the downward
closure of $s(F)=\{(a_1,b_1),\ldots,(a_l,b_l)\}$ where $a_i$ is in
$\mathbb{N}^{m-1}$ and $b_i$ is in $\mathbb{N}$ and each $(a_i,b_i)$ is maximal
with respect to the pointwise ordering.  Let $ord(s(F))$ be the natural sum
over $1\leq i\leq l$ of the terms $\om^{ord(a_i)}\cdot b_i$.  Let
$ord(F):=1+ord(s(F))$. If $F$ is empty then $ord(F):=0.$ (Note that the singleton set consisting of the zero
vector describes the second minimal element.) We prove by induction on the
cardinality of $s(G)$ that $F\subseteq G$ implies $ord(F)\leq ord(G)$; the
proposition then follows from this.

So assume $\emptyset\not=F\subseteq G$ and assume that $s(F)=\{(a_1,b_1),\ldots,(a_k,b_k)\}$ and $s(G)=\{(c_1,d_1),\ldots,(c_l,d_l)\}$.
Let $S_1:=\{(a,b)\in s(F):  (a,b)\leq(c_1,d_1)\}$ and $S_2:=s(F)\setminus S_1$.
Then $S_2\subseteq s(G)\setminus\{(c_1,d_1)\}$ and by induction hypothesis we may assume that $ord(s(S_2))\leq ord(s(G)\setminus\{(c_1,d_1)\})$ if $S_2$ is not empty.
It thus suffices to show $ord(S_1)\leq ord(\{(c_1,d_1)\})$.
If $S_1$ is a singleton then the assertion follows easily. 
Problems might occur when $S_1$ is not a singleton because $ord(\{(c_1,d_1)\})$ is in general not additively closed.
We may assume after renumbering that $S_1=\{(a_1,b_1),\ldots,(a_n,b_n)\}$.
Assume that there is an $(a_i,b_i)\in S_1$ such that $a_i=c_1$. (The case that $a_i\not=c_1$ for all $i$ is similar but easier.)
Then $b_i=d_1$ is excluded because if $(a_j,b_j)\in S_1$ is another element then
$(a_j,b_j)\leq (c_1,d_1)=(a_i,b_i)$ and $(a_j,b_j)$ would not be maximal. Therefore $b_i<d_1$. 
Now pick any $(a_j,b_j)\in S_1$ different from $(a_i,b_i)$. Then $a_j=a_i$ is impossible since then either $(a_i,b_i)$ is not maximal if $b_i<b_j$ or 
$(a_j,b_j)$ is not maximal if $b_j<b_i$. Since $(a_j,b_j)\leq  (c_1,d_1)$ we conclude $a_j\leq c_1=a_i$. Since $a_j\not=c_1$ we conclude that
$a_j$ is lexicographically smaller than $c_1$ so that $ord(a_j)<ord(c_1)$.
This means that all such terms $(a_j,b_j)$ get assigned ordinals $\om^{ord(a_j)}\cdot b_j<\om^{ord(c_1)}$.
Summing up all terms for elements in $S_1$ we get a strict upper bound provided by $\om^{ord(c_1)}\cdot b_i+\om^{ord(c_1)}\leq \om^{ord(c_1)}\cdot d_1=ord(\{(c_1,d_1)\})$.
\end{proof}

Combining Corollary \ref{upbd} and Proposition \ref{lobd} now yields Theorem~\ref{mk} for $k=1$.

\subsection{Bounded lower sets in $\mathbb{N}^m\times k$}

Before we move on to $I(\mathbb{N}^m)$, let's briefly consider
$D(\mathbb{N}^m\times k)$.  We stated the type of this in Theorem~\ref{mk}.  In
this subsection we prove it.  First some notation:

\begin{notn}
For $X$ a partially-ordered set and $x\in X$, we define the upward closure
$X^{\geq x}$ to be $\{y\ge x: y\in X\}$; this is the
smallest upward closed subset of $X$ containing $x$.  
\end{notn}

Now the proof:

\begin{proof}[Proof of Theorem~\ref{mk} for $k\geq 1$.]
To prove the upper bound, note that 
there's an obvious embedding of
$D(\mathbb{N}^m \times k)$ into $D(\mathbb{N}^m)^k$, by mapping
\[
S\mapsto (S\cap (\mathbb{N}^m\times\{0\}),\ldots,S\cap
(\mathbb{N}^m\times\{k-1\})),
\]
so \[ o(D(\mathbb{N}^m \times k))\le o(\mathbb{N}^m)^{\otimes
k}=\omega^{\omega^{m-1}k}.\]
This leaves the lower
bound.  For this, we induct on $k$.  The case $k=1$ has already been proven
above, so that leaves the inductive step.

We will construct a total order extending $D(\mathbb{N}^m\times k)$ that has the
required order type.  First, choose a total order extending
$D(\mathbb{N}^m\times\{k-1\})$ of order type $\omega^{\omega^{m-1}}$; this is
possible
by the above.  We will sort the elements $S$ of $D(\mathbb{N}^m\times k)$ first
by the value of $S\cap (\mathbb{N}^m\times\{k-1\})$ (according to this order),
and then find some way to break the ties.

So consider some element $T\in D(\mathbb{N}^m\times \{k-1\})$ and consider the
set $P_T$ of $S\in D(\mathbb{N}^m\times k)$ such that $S\cap
(\mathbb{N}^m\times\{k-1\})=T$.  What is the maximum extending ordinal of this
set?  To answer this, observe that there is some element $x\in \mathbb{N}^m$
such that $(x,k-1)\notin T$.  So in fact $X^{\geq x}\times \{k-1\}$ is disjoint from
$T$; and $X^{\geq x}$ is isomorphic to $\mathbb{N}^m$.  This gives us an inclusion of
$D(\mathbb{N}^m \times (k-1))$ into $P_T$, so $o(P_T)$ is (by the induction
hypothesis) at least $\omega^{\omega^{m-1}(k-1)}$.

Therefore $o(D(\mathbb{N}^m\times k))\ge
\omega^{\omega^{m-1}(k-1)} \omega^{\omega^{m-1}} = \omega^{\omega^{m-1}k}$.
This completes the proof.
\end{proof}

\section{General lower sets in $\mathbb{N}^m$}

In this section we compute $o(I(\mathbb{N}^m))$.  As we will see,
$I(\mathbb{N}^m)\setminus \{\mathbb{N}^m\}$ can be approximately decomposed as
a product over nonempty
$C\subseteq\{1,\ldots,m\}$ of $D(\mathbb{N}^C)$; however, the exact nature of
this decomposition will be slightly different in the upper bound proof and in
the lower bound proof.

\subsection{The upper bound proof}

In this section we prove a proposition that expresses one half of this
decomposition.

We will need the following lemma, which is an easy consequence of some known
facts:

\begin{lem}
\label{rectangles}
Let $P = \alpha_1 \times \ldots \times \alpha_m$ be a finite Cartesian product
of well-orders.  Then any lower set of $P$ is a finite union of rectangles
$\beta_1 \times \ldots \times \beta_m$ for some $\beta_i\le\alpha_i$.
\end{lem}

\begin{proof}
In general, a lower set in a well partial order is a finite union of
\emph{ideals}, which is a downward-closed set $I$ with the additional property
that if $x,y\in I$, there exists $z\ge x,y$ with $z\in I$; one may see
e.g.~\cite{ideals} for a proof, where this is a combination of Lemma~2.6 and
Proposition~2.10.  Moreover, the ideals of $X\times Y$ are precisely the sets
$I\times J$ where $I$ is an ideal of $X$ and $J$ is an ideal of $Y$; again one
may see \cite{ideals}, where this appears as Proposition~4.8.  Since obviously
an ideal of $\alpha_i$ is an ordinal $\beta_i\le\alpha_i$, the result follows.
\end{proof}

We also define $I_0(X):=I(X)\setminus\{X\}$.

In the following proposition we assume $m\geq 1$. We recall that for a partially ordered set $S$, the set $I(S)$ of lower subsets of $S$ is a lattice of subsets of $S$, and if $S$ is directed, then $I_0(S)$ is a sublattice of $I(S)$. Moreover, if $S$ is a lattice, then $D(S)$ is a sublattice of $I(S)$. In particular, $I(\omega^m)$ has the sublattices $D(\omega^m) \subseteq I_0(\omega^m)$.
We let $i = (i_1,...,i_n)$ range over all sequences $1 \leq  i_1 < \cdots < i_n \leq  m$  where $n \geq 1$, and for such $i$ we put $\vert i\rvert:= n$ and define the restriction map 
$\pi_i: \omega^m \to \omega^{\vert i\lvert}$ by $\pi_i(a) := (a_{i_1},...,a_{i_n})$ for $a = (a_1,...,a_m) \in \omega^m$. Consider the lattice
$D := \prod_i D(\omega^{\lvert i\rvert})$
where
$S \vee T = (S_i \cup T_i)$,
$S \wedge T = (S_i \cap T_i$) for $S=(S_i), T = (T_i) \in D$.
The lattice morphism
$S \mapsto \pi_i^{-1}(S): I(\omega^{\lvert i\rvert}) \to I(\omega^m)$
restrict to lattice morphisms
$S \mapsto \pi_i^{-1}(S): D(\omega^{\lvert i\rvert}) \to I_0(\omega^m)$
which combine to a map
$S = (S_i)\mapsto  \varphi(S) := \bigcup_i \pi_i^{-1}(S_i): D \to I_0(\om^m)$
satisfying $\varphi (S\vee T)=\varphi(S)\cup \varphi (T)$ for $S,T \in D$.

\begin{prop}
\label{puttogetherlem}
The map $\varphi$ is increasing and onto.
\end{prop}
\begin{proof}
The first statement is clear by the remark before the proposition. To show surjectivity, by the previous corollary it is enough to show that each rectangle 
$I = \alpha_1 \times  \cdots \times \alpha_m \in I_0(\omega^m)$, where $\alpha_1,...,\alpha_m \leq \om$, is of the form $\varphi(S)$ for some $S \in D$. 
For this let $j = (j_1,\ldots,j_n)$ where $1\leq j_1 <\cdots<j_n \leq m$ are the indices $j \in\{ 1,...,m\}$ with $\alpha_j < \omega$. 
Then $n\geq 1$, and we have
$I=\varphi(S)$ for $S=(S_i)\in D$ given by $S_j :=\pi_j(I)$ and $S_i :=\emptyset$ for $i\not=j$.
\end{proof}

Thus we can conclude the upper bound:

\begin{thm}
\label{upbdunbd}
\[ o(I(\mathbb{N}^m)) \le
\omega^{\sum_{k=1}^{m} \omega^{m-k}\binom{m}{k-1} }+ 1. \]
\end{thm}

\begin{proof}
Applying Proposition~\ref{puttogetherlem} with $\alpha_i=\omega$ for all $i$,
together with Theorem~\ref{mk}, yields that
\[ o(I_0(\mathbb{N}^m)) \le \omega^{\sum_{k=1}^{m} \omega^{m-k}\binom{m}{k-1} };
\]
since $I(\mathbb{N}^m)=I_0(\mathbb{N}^m)\cup\{\mathbb{N}^m\}$, we conclude
\[ o(I_0(\mathbb{N}^m)) \le \omega^{\sum_{k=1}^{m} \omega^{m-k}\binom{m}{k-1} }
+1 .
\]

\end{proof}

\subsection{The lower bound proof}

For the proof of the lower bound, we will need some additional definitions.
Rather than deal with fully specified lower sets in $I(\mathbb{N}^m)$, we will
also define ``partial specifications'' of such sets.

\begin{defn}
Given a function $f:S\to T$ and $A\subseteq S$, define the
``intersection image'' $\overline{f}(A)$ to be $T\setminus f(S\setminus A)$,
or equivalently to be $\{p \in T: f^{-1}(p)\subseteq A\}$.
\end{defn}

\begin{defn}
A \emph{partial specification} $X$ on $\mathbb{N}^m$ consists of a nonempty
collection $\mathscr{C}$ of subsets of $[m]$ and, for each
$C\in\mathscr{C}$, some $X_C\in I_0(\mathbb{N}^C)$, such that:
\begin{itemize}
\item if $D\subseteq[m]$ and $C\in\mathscr{C}$ with $|D|<|C|$, then
$D\in\mathscr{C}$, and
\item if $D\subseteq C\in\mathscr{C}$, then $X_D=\overline{\pi}_D(X_C)$. (Here $\pi_D=\pi_{C,D}:\mathbb{N}^C\to \mathbb{N}^D$ denotes the restriction map.)
\end{itemize}

Given a partial specification $X$ on $\mathbb{N}^m$ and a set $S\in
I_0(\mathbb{N}^m)$, we will say that $S$ is \emph{compatible} with $X$ if
$\overline{\pi}_C(S)=X_C$ for each $C\in\mathscr{C}$.  We define $\mathscr{A}_X$
to be the set of all $S\in I_0(\mathbb{N}^m)$ compatible with $X$.
\end{defn}

Observe that, for a partial specification $X$ with domain $\mathscr{C}$, if
$X_C$ is known for all maximal elements $C\in\mathscr{C}$ (under inclusion),
then $X_D$ is known for all $D\in\mathscr{C}$.

We will show here how to get a lower bound on $o(\mathscr{A}_X)$ for any partial
specification $X$, based only on the domain of $X$.  Then, to get a lower bound
on $o(I_0(\mathbb{N}^m))$, we need only take $X$ to be the unique partial
specification on $\mathbb{N}^m$ with domain $\{\emptyset\}$, since every proper
lower set in $\mathbb{N}^m$ is compatible with this specification.  (Conversely,
if the domain of $X$ is $\wp([m])$, then $X_{[m]}$ is the
unique element of $I_0(\mathbb{N}^m)$ that is compatible with $X$.)

With both the components of the upper and lower bounds laid out, we can now
prove the theorem.

\begin{prop}
\label{lowerboundprop}
Let $X$ be a partial specification on $\mathbb{N}^m$ with domain
$\mathscr{C}$.  Then
\[o(\mathscr{A}_X)\ge\bigotimes_{C\notin\mathscr{C}} \omega^{\omega^{|C|-1}}.\]
\end{prop}

In fact, by the arguments above, this lower bound will actually be an equality,
but we only care about the lower bound.  Note $o(\mathscr{A}_X)$ increases as
the domain of $X$ gets smaller; the less-specified $X$ is, the more sets are
compatible with it.

\begin{proof}
We prove this by downard induction on the size of the domain.  It's trivally
true for any partial specification $X$ on $\mathbb{N}^m$ with domain
$\wp([m])$, since in
this case one will have $|\mathscr{A}_X|=1$ and the product will be $1$ as well.
So suppose $\mathscr{C}\subseteq[m]$ is a valid domain for a partial
specification and that the proposition
holds for all partial specifications on $\mathbb{N}^m$ with that domain.
Pick some $C\in\mathscr{C}$ of maximum cardinality; we want to
show the statement holds for any partial specification with domain
$\mathscr{C}\setminus \{C\}$.

So let $X$ be a partial specification with domain $\mathscr{C}\setminus \{C\}$.  We
want to put a total order on $\mathscr{A}_X$ in order to get a lower bound on
$o(\mathscr{A}_X)$.  Given any $S\in \mathscr{A}_X$, we can obtain a partial
specification $Y$ with domain $\mathscr{C}$ by taking $Y_D=\overline{\pi}_D(S)$
for $D\in\mathscr{C}$; observe then that $S\in \mathscr{A}_Y$.  Obviously, any
such $Y$ has $Y_D=X_D$ for any $D\ne C$; the only distinguishing feature of $Y$
is the value of $Y_C$.

Note that not every $T\in I_0(\mathbb{N}^C)$ is a possible value of
$Y_C$, since if $T=Y_C$ we have the restriction that for $D\subseteq C$ we have
$\overline{\pi}_D(T)=X_D$.  But given such a $T$ we can define $Y(T)$ to be $Y$
obtained by setting $Y_C=T$.  So we will put a total order on $\mathscr{A}_X$ by
first putting a total order on the set of such $T$ (call this set
$\mathscr{T}$), and sorting elements $S$ of $\mathscr{A}_X$ by the value of
$\overline{\pi}_C(S)$; and then, for each such $T$, putting a total order on
$\mathscr{A}_{Y(T)}$.  So we will get a lower bound on $o(\mathscr{A}_X)$ of the
form $\sum_{T\in\mathscr{T}} o(\mathscr{A}_{Y(T)})$ (using the total order on
$\mathscr{T}$ that we have picked).

In fact, by the inductive hypothesis, for any $T\in\mathscr{T}$, we know that
\[o(\mathscr{A}_{Y(T)})\ge\bigotimes_{D\notin\mathscr{C}}
\omega^{\omega^{|D|-1}}.\]
Thus, we immediately get that
\[o(\mathscr{A}_X)\ge (\bigotimes_{D\notin\mathscr{C}} \omega^{\omega^{|D|-1}})
	o(\mathscr{T}).\]

It then remains to show that $o(\mathscr{T})\ge\omega^{\omega^{|C|-1}}$.  Once
we know this, we will have
\[o(\mathscr{A}_X)\ge\bigotimes_{D\notin\mathscr{C}\setminus\{C\}} \omega^{\omega^{|D|-1}},\]
because, by assumption, $|C|\le|D|$ for any $D\notin\mathscr{C}$, and so the
ordinary product here coincides with the natural product.

So let $A\subseteq \mathbb{N}^C$ be defined by $A=\cup_{D\subsetneq
C} \pi_D^{-1}(X_D)$.  Then for any $V\in D(\mathbb{N}^C)$, $A\cup
V\in\mathscr{T}$.  Pick some $b\in \mathbb{N}^C\setminus A$, and
consider $X^{\geq b}:=({\mathbb{N}^C})^{\geq b}$.
Given $V\in D(X^{\geq b})$ let $L(V)$ be the downward closure of $V$ in
$\mathbb{N}^C$.  Observe that the map from $D(X^{\geq b})$ to
$\mathscr{T}$ given by $V\mapsto A\cup L(V)$ is injective and indeed an
embedding.  Also observe that $X^{\geq b}$ is isomorphic to $\mathbb{N}^{C}$.  So by
Theorem~\ref{mk}, $o(D(X{\geq b})=\omega^{\omega^{|C|-1}}$, and so
$o(\mathscr{T})\ge\omega^{\omega^{|C|-1}}$, as needed.  This completes the
proof.
\end{proof}

We can now prove the lower bound:

\begin{thm}
\label{lowbdunbd}
\[ o(I(\mathbb{N}^m)) \ge
\omega^{\sum_{k=1}^{m} \omega^{m-k}\binom{m}{k-1} }+ 1. \]
\end{thm}

\begin{proof}
Let $X$ be the unique partial specification on
$\mathbb{N}^m$ with domain $\{\emptyset\}$; then
$I_0(\mathbb{N}^m)=\mathscr{A}_X$.  By Proposition~\ref{lowerboundprop}, then,
\[ o(I_0(\mathbb{N}^m)) = o(\mathscr{A}_X) \ge
\omega^{\sum_{k=1}^{m} \omega^{m-k}\binom{m}{k-1} }. \]
Therefore
\[ o(I(\mathbb{N}^m)) \ge
\omega^{\sum_{k=1}^{m} \omega^{m-k}\binom{m}{k-1} } + 1, \]
proving the theorem.
\end{proof}

\subsection{Putting together the proof}
\begin{proof}
Theorem ~\ref{omegaunbounded} now follows from Theorems~\ref{upbdunbd} (for the upper bound) and 
~\ref{lowbdunbd} (for the lower bound).
\end{proof}

\section{Application to monomial ideals}
\label{secapp}

We now discuss applications to computational complexity and provide complementary results to Corollary 3.27 in \cite{AP}.
In the sequel we work with ordinals below $\om^{\om^\om}$. 
For these ordinals we consider the Hardy functions $H_\alpha:\Nat\to\Nat$, defined recursively as follows.
Let $H_0(x):=x$, $H_{\alpha+1}(x):=H_\alpha(x+1)$ and for a limit $\lambda$ let
$H_{\lambda}[x]:=H_{\lambda[x]}(x+1)$ where $\lambda[x]$ denotes the $x$-th member of the canonical fundamental sequence for $\lambda$. 
These fundamental sequences are defined by recursion as follows.
If $\lambda=\om^{\lambda'}$ with $\lambda'$ a limit then $\lambda[x]=\om^{\lambda'[x]}$. 
If $\lambda=\om^{\beta+1}$ then $\lambda[x]=\om^{\beta}\cdot x$. 
If $\lambda=\om^{\beta}+\lambda'$ with $\lambda'<\lambda$ a limit  then $\lambda[x]=\om^{\beta}+\lambda'[x]$. 

For technical reasons we also define $(\alpha+1)[x]:=\alpha$.

In the sequel we stick for simplicity to the case $k=2$.
We believe that the case of more than two factors can be carried out analogously.

By standard results (see, for example, Lemma 4 in \cite{BCW}) it is known that
$H_{\om^\om}$ is a variant of the non primitive recursive Ackermann function, and $H_{\om^{\om+2}}$ is
roughly the result of iterating the Ackermann function twice.
We have shown in section three that $o(I(\Nat^2))=\om^{\om+2}$. In this section we show that $H_{\om^{\om+2}}$
bounds the lengths of effectively given bad sequences in $I(\Nat^2)$. We also show a corresponding result for monomial ideals
of a polynomial ring over a field with two variables.

Our proof for $I(\Nat^2)$ highlights in particular how the unbounded downward closed subsets are responsible for the addition of two 
in the ordinal bound. Moreover it highlights the rule of thumb that the Hardy functions indexed by ordinals below the type of 
a well partial order describe the complexity of effectively given bad finite sequences.

Let us define a complexity measure for downward closed sets in $\Nat^k$.
For $k=1$ and finite $\al$ put $ M\al=\al$ and for $\al=\Nat$ put $M\al=0$. This measure is extended to cartesian products of initial segments as follows:
put $M(\al_1\times \cdots \times \al_k):=\max\{M\al_i:i\leq k\}$. Almost the same measure can be applied when dealing with monomial ideals 
which can be identified with upward closed sets in $\Nat^k$ (see the next theorem).

If a downward closed set $D$ is a shortest finite union of $k$-times cartesian products of initial segments $J_i$  (this means that the number of products used in the representation of $D$ is minimal) then we put
$MD:=\max\{M(J_i)\}$. Then for any natural number $d$ there will only be finitely many downward closed sets of complexity not exceeding $d$.

\begin{Proposition}\label{twobound}
For each $K\in\Nat$ there are downward closed sets $D_1,\ldots,D_L$ of $\Nat^2$ such 
that $L\geq H_{\om^{\om+2}}(K)-K$,  $M(D_i)\leq (K+i)^2$ for $1\leq i\leq L$, and $D_i\not\subseteq  D_j$ for $1\leq i<j\leq L$.
\end{Proposition} 

\begin{proof}
Let $\al_0:=\om^{\om+2}$ and let $\al_{i+1}:=\al_i[K+i]$.
Then $\alpha_i>0$ yields $\alpha_i>\alpha_{i+1}$ and moreover we find
\[H_{\al_0}(K)=H_{\al_0[K]}(K+1)=\ldots=H_{\al_0[K][K+1]\ldots[K+L-1]}(K+L)=K+L\] where $L$ is minimal with $\al_L=0$.

For $\al=\om^{\om+1}\cdot p+\om^\om\cdot q+\om^{a_1}\cdot b_1+\ldots+\om^{a_r}\cdot b_r$  in normal form where $p,q,r\geq 0$ and $a_1>\ldots>a_r$ and
$b_1,\ldots,b_r>0$ let
\[N\al:=p+q+b_1+\cdots+b_r+\max\{a_1,\ldots,a_r\}.\]
Then an induction on $i$ yields $N\al_i\leq (K+i)^2$.

For a set $S\subseteq \Nat^2$ let $S_\leq$ be the least downward closed set containing $S$.
For $\al=\om^{\om+1}\cdot p+\om^\om\cdot q+\om^{a_1}\cdot b_1+\ldots+\om^{a_r}\cdot b_r$
in normal form
define a downward closed set $D(\al)$ as follows:
\[
D(\al):=(p\times \Nat)\cup (\Nat\times q)\cup\{(p+a_1+1,q+b_1),\ldots,(p+a_r+1,q+b_1+\ldots+b_r)\}_{\leq}.
\]

Assume that $\al'=\om^{\om+1}\cdot p'+\om^\om\cdot q'+\om^{a_1'}\cdot b_1'+\ldots+\om^{a'_{r'}}\cdot b'_{r'}$
is in normal form 
and assume $\al'<\al$. We show that $D(\al)$ is not a subset of $D(\al')$.
The proof can be established by a simple case distinction.

Case 1. $p'<p$. Then $p\times\Nat$ is not contained in $D(\al')$.

Case 2. $p'=p$ and $q'<q$. Then $\Nat\times q$ is not contained in $D(\al')$.

Case 3. $p=p'$, $q=q'$ and there exists a $j_0$ such that $a'_{j_0}<a_{j_0}$ or $(a_{j_0}=a'_{j_0}$ and $b_{j_0}'<b_{j_0}$) and for all $l<j_0$ we have
$a_l=a'_l$ and $b_l=b'_l$.

Then $ \{(p+a_{j_0}+1,q+b_1+\cdots +b_{j_0}\}_{\leq}$ is not contained in $D(\al')$.
This can be checked by verifying that $(p+a_{j_0}+1,q+b_1+\cdots +b_{j_0})$ is in no interval showing up in the representation of $D(\al')$.
The first two intervals are left out since $p<p+a_{j_0}+1$ and $q<q+b_1+\cdots +b_{j_0}$.
The intervals with index $i> j_0$ do not contain $p+a_{j_0}+1$ in their left coordinates and the intervals with index 
$i<j$ do not contain $q+b_1+\cdots +b_{j_0}$ in their right coordinates.
A similar argument applies for $i=j_0$.

The result follows by putting $D_i:=D(\al_i)$ for $i\geq 1.$
\end{proof}

Let us now consider polynomial rings in two variables $X,Y$ over a field $F$.  
We believe that the case of more then two variables can be carried out analogously.

The
degree of a monomial ideal is the maximum degree of the minimal generating set
of monomials.  We denote by $(m_1,\ldots,m_l)$ the monomial ideal generated by
the monomials $m_i$. The degree of a monomial ideal with minimal representation $(m_1,\ldots,m_l)$ is equal to $\max\{\deg(m_i)\}$

\begin{thm}
 For each $K\in\Nat$ there are monomial ideals $I_1,\ldots,I_L$ of  $F(X,Y)$ such 
that $L\geq H_{\om^{\om+2}}(K)-K$, $\deg(I_i)\leq (K+i)^2$ for
$1\leq i\leq L$ and $I_j\not\subseteq I_i$ for $1\leq i<k\leq L$.\end{thm}
\begin{proof}
We associate to an ordinal $\alpha<\om^{\om+2}$ a monomial ideal $I(\alpha)$ of $F[X,Y]$ such that for the descending sequence $(\al_i)$ 
of ordinals constructed in the proof of the last lemma we have $\deg(I(\alpha_i))\leq (K+i)^2$ for $i\geq 1$ and $I(\alpha_j)\not\subseteq I(\alpha_i)$ for $i<j$.

Assume that $\al=\om^{\om+1}\cdot p+\om^\om\cdot q+\om^{a_1}\cdot b_1+\ldots+\om^{a_{r}}\cdot b_{r}$
is in normal form.
Let $c_j:=b_1+\cdots+b_j$.
Let $$I(\al):=(X^{a_1+p+1}Y^{q},X^{a_2+p+1}Y^{c_1+q+1},\ldots,X^{a_r+p+1}Y^{c_{r-1}+q+1},X^pY^{c_{r}+q+1}).$$
For $r=0$ we put $I(\al):=(X^{p+1}\cdot X^q,X^p\cdot Y^{q+1})$.

Assume that $\al'=\om^{\om+1}\cdot p'+\om^\om\cdot q'+\om^{a_1'}\cdot b_1'+\ldots+\om^{a'_{r'}}\cdot b'_{r'}$ 
is in normal form and assume $\al'>\al$.
Let $c'_j:=b'_1+\cdots+b'_j$.
We show that $I(\al)$ is not a subset of $I(\al')$.

The proof can be established by a simple case distinction.

Case 1. $p<p'$. Then $X^pY^{c_{r}+q+1}$ is not an element of  $I(\al')$ (even for $r=0$)
since all generators of $I(\al')$ contain a multiple of $X^{p'}$.

Case 2. $p=p'$ and $q<q'$. Then $X^{a_1+p+1}Y^q$ (or $X^{p+1}Y^q$ in the case $r=0$) is not an element of  $I(\al')$
since all generators of $I(\al')$ contain a multiple of $Y^{q'}$.

Case 3. $p=p'$, $q=q'$ and there exists a $j_0$ such that $a_{j_0}<a_{j_0}'$ or $(a_{j_0}=a'_{j_0}$ and $b_{j_0}<b'_{j_0}$) and for all $l<j_0$ we have
$a_l=a'_l$ and $b_l=b'_l$.

Case 3.1.  $a_{j_0}<a'_{j_0}$. Then $X^{a_{j_0}+p+1}Y^{c_{j_0-1}+q+1}\not \in I(\al')$. Indeed, for $i<j_0$ we have $a_{j_0}<a_i=a_i'$ and hence 
\[X^{a_{j_0}+p+1}Y^{c_{j_0-1}+q+1}\not \in(X^{a'_i+p+1}Y^{c'_{i-1}+q+1}).\] If $i=j_0$ then from $a_{j_0}<a_{j_0}'$ we conclude 
\[X^{a_{j_0}+p+1}Y^{c_{j_0-1}+q+1}\not \in(X^{a'_{j_0}+p+1}Y^{c'_{j_0-1}+q+1}).\]
For $i>j_0$ 
we find $X^{a_{j_0}+p+1}Y^{c_{j_0-1}}\not \in(X^{a'_i+p+1}Y^{c'_{i-1}+q+1})$ because $c_{j_0-1}=c'_{j_0-1}<c'_{i-1}$.

Case 3.2. $(a_{j_0}=a'_{j_0}$ and $b_{j_0}<b'_{j_0}$).
Then $X^{a_{j_0+1}+p+1}Y^{c_{j_0}+q+1}\not \in I(\al')$. Indeed, for $i<j_0$ we obtain $X^{a_{j_0+1}+p+1}Y^{c_{j_0}+q+1}\not \in (X^{a'_{i+1}+p+1}Y^{c'_i+q+1})$ since $a_{i+1}'=a_{i+1}\geq a_{j_0}>a_{j_0+1}$.
For $i\geq j_0$ we conclude \[X^{a_{j_0+1}+p+1}Y^{c_{j_0}+q+1}\not \in
(X^{a'_{i+1}+p+1}Y^{c'_i+q+1})\] since $c_{j_0}<c'_{j_0}\leq c_i'$.
\end{proof}

{\bf Alternative proof of Proposition 4.1:} The referee pointed out that Proposition 4.1 can be deduced from Theorem 4.2 in the following very elegant way: Given a monomial ideal $I$ of $F[X,Y]$ let 
$$D(I):=\{(i,j)\in \Nat^2: X^iY^j\not\in I\},$$ a downward closed subset of $\Nat^2$ satisfying $M(D(I))\leq \deg(I)$. If for a given value $K\in \Nat$ we have $I_1,\ldots,I_L$ as in Theorem 4.2, then $D(I_1),\ldots,D(I_L)$ are downward closed subsets of $\Nat^2$ with the properties required in 4.1.\hfill $\Box$

The lower bound provided by previous theorem is
essentially sharp  in the sense that $K\mapsto L$ depends elementary recursively  on $H_{\al}$ where $\al$
is the maximal order type under consideration.
This can be shown by a reification analysis using the results on the upper bound for the maximal order type involved.
For this one can exploit that the lengths of elementary descending sequences of ordinals can be bounded in terms of the Hardy functions
as shown for example in \cite{BCW}.

\subsection*{Acknowledgements}  Thanks to David Belanger and to the anonymous
referee for helpful comments.

\end{document}